\documentclass[12pt]{amsart}
\usepackage{amsopn}
\usepackage{amssymb}
\usepackage{graphicx}
\usepackage[all]{xy}
\usepackage{amscd}
\usepackage{color}
\usepackage[colorlinks]{hyperref}
\newtheorem{thm}{Theorem}[section]

\newtheorem{prop}[thm]{Proposition}

\newtheorem{defn}[thm]{Definition}

\theoremstyle{definition}
\theoremstyle{remark}

\begin{document}
\title[K-G-frames in Hilbert spaces]
{ $K$-$G$-frames and approximate $K$-$G$-duals in Hilbert spaces }

\author{Jahangir Cheshmavar}
\author{Maryam Rezaei Sarkhaei}

\email{j$_{_-}$cheshmavar@pnu.ac.ir}
\email{rezaei.sarkhaei@yahoo.com}

\address{Department of Mathematics,
Payame Noor University, P.O.BOX 19395-3697, Tehran, IRAN.}

\thanks{}
\thanks{}

 \dedicatory{} \commby{}
% ----------------------------------------------------------------
\begin{abstract}
Approximate duality of frame pairs have been investigated by
Christensen and Laugesen in (Sampl. Theory Signal Image Process.,
9, 2011, 77-90), with the motivation to obtain an important
applications in Gabor systems, wavelets and general frame theory.
In this paper we obtain some of the known results in approximate
duality of frames to $K$-$g$-frames. We also obtain new
$K$-$g$-frames and approximate $K$-$g$-duals from a $K$-$g$-frame
and an approximate $K$-$g$-dual. Finally, we give an equivalent
condition under which the subsequence of a $K$-$g$-frame still to
be a $K$-$g$-frame.
\end{abstract}
\maketitle
% ----------------------------------------------------------------

\section{\bf Introduction and Preliminaries}\label{introduction}

The concept of frames in Hilbert spaces were introduced in the
paper ~\cite{Duffin.Schaeffer} by Duffin and Schaeffer to study
some problems in nonharmonic Fourier series, reintroduced by
Daubechies et al. in ~\cite{Daubechies.Grossmann.Meyer} to study
the connecting with wavelet and Gabor systems. For special
applications, various generalization of frames were proposed, such
as frame of subspace ~\cite{Cazassa.Kutyniok}, $g$-frames
~\cite{Sun}, $K$-frames ~\cite{Gavruta} by G$\breve{a}$vruta to
study the atomic systems with respect to a bounded linear operator
$K$ in Hilbert spaces. The concept of $K$-$g$-frames which more
general than ordinary g-frames considered by Authors in
~\cite{Asgari.Rahimi,Zhou.Zhu.1,Zhou.Zhu.2}. After that, some
properties of $K$-frames was extended to $K$-$g$-frames in
\cite{Hua.Huang} by Hua and Huang.

One of the main reason for considering frames and any type of
generalization of frames, is that, they allows each element in the
space to be non-uniquely represented as a linear combination of
the frame elements, by using of their duals; however, it is
usually complicated to calculate a dual frame explicitly. For
example, in practical, one has to invert the frame operator, in
the canonical dual frames, which is difficult when the space is
infinite-dimensional. One way to avoid this difficulty is to
consider approximate duals. The concepts of approximately dual
frames have been studied since the work of Gilbert et al.
\cite{Gilbert.et.al} in the wavelet setting, see for example
Feichtinger et al. \cite{Feichtinger.Kaiblinger} for Gabor systems
and reintroduced in systematic by Christensen and Laugesen in
\cite{Christensen.Laugesen} for dual frame pairs.

In this paper, the advantage of $K$-$g$-frames in
comparison of $g$-frames is given; with this motivation, we obtain new
$K$-$g$-frames and approximate $K$-$g$-duals and derive some of results
for the approximate duality of $K$-$g$-frames and their redundancy.

What we discuss in the following sections; we will review
some notions relating to frames, $K$-frames and $K$-$g$-frames
in the rest of this section. New $K$-$g$-frames and the advantage of $K$-$g$-frames are found
in Section 2. In Section 3 we define approximate duality of
$K$-$g$-frames and get some important properties of approximate
$K$-$g$-duals, also we extend some results of approximate duality
of frame to $K$-$g$-frames based in the paper
\cite{Christensen.Laugesen}. Section 4 contains two results on the
redundancy.

Throughout this paper, $J$ is a subset of integers $\mathbb{Z}$,
$\mathcal{H}$ is a separable Hilbert space,
$\{\mathcal{H}_j\}_{j\in J}$ is a sequence of closed subspaces of
$\mathcal{H}$. Let $\mathcal{B}(\mathcal{H},\mathcal{H}_j)$ be the
collection of all bounded linear operators from $\mathcal{H}$ into
$\mathcal{H}_j$, $\mathcal{B}(\mathcal{H},\mathcal{H})$ is denote
by $\mathcal{B}(\mathcal{H})$; for $K \in
\mathcal{B}(\mathcal{H})$, $\mathcal{R}(K)$ is the range of $K$
and also ${I}_{\mathcal{R}(K)}$ is the identity operator on
$\mathcal{R}(K)$. The space $l^2 \left( \{\mathcal{H}_j\}_{j\in J}
\right)$ is defined by
\begin{eqnarray}
l^2 \left ( \{\mathcal{H}_j\}_{j\in J} \right
)=\left\{\{f_j\}_{j\in J}:\,\ f_j \in \mathcal{H}_j,
\|\{f_j\}_{j\in J}\|^2=\sum_{j\in J}\|f_j\|^2 <+\infty \right \} ,
\end{eqnarray}
with inner product given by
\begin{eqnarray}
\langle \{f_j\}_{j\in J}, \{g_j\}_{j\in J} \rangle=\sum_{j\in
J}\langle f_j, g_j \rangle.
\end{eqnarray}
Then $l^2 \left( \{\mathcal{H}_j\}_{j\in J} \right)$ is Hilbert
space with pointwise operations.\\
Next, some terminology relating to Bessel and $g$-Bessel systems,
frames, $g$-frames, $K$-frames and related notions.

\bigskip
A sequence $\{f_j\}_{j\in J}$ contained in $\mathcal{H}$ is called
a Bessel system for $\mathcal{H}$, if there exists a positive
constant $B$ such that, for all $f \in \mathcal{H}$
\begin{eqnarray}
\sum_{j \in J}|\langle f, f_j \rangle|^2\leq B\|f\|^2.
\end{eqnarray}
The constant $B$ is called a Bessel bound of the system. If, in
addition, there exists a lower bound $A>0$ such that, for all $f
\in \mathcal{H}$
\begin{eqnarray}
A\|K^{\ast}f\|^2\leq \sum_{j \in J}|\langle f, f_j \rangle|^2,
\end{eqnarray}
the system is called a $K$-frame for $\mathcal{H}$. The constant
$A$ and $B$ are called $K$-frame bounds.\\

\noindent {\bf Remark 1:}
If $K=I_{\mathcal{H}}$, then $K$-frames are called the ordinary frames.

\bigskip
A sequence $\{\Lambda_j \in
\mathcal{B}(\mathcal{H},\mathcal{H}_j): j\in J\}$ is called a
$g$-Bessel system for $\mathcal{H}$ with respect to
$\mathcal{H}_{j}$ if there exists a positive constant $B$ such
that, for all $f \in \mathcal{H}$
\begin{eqnarray}
\sum_{j \in J}\|\Lambda_jf\|^2\leq B\|f\|^2.
\end{eqnarray}
The constant $B$ is called a $g$-Bessel bound of the system. If,
in addition, there exists a lower bound $A>0$ such that, for all
$f \in \mathcal{H}$
\begin{eqnarray}
A\|f\|^2\leq \sum_{j \in J}\|\Lambda_jf\|^2,
\end{eqnarray}
the system is called a $g$-frame for $\mathcal{H}$ with respect to
$\{\mathcal{H}_j\}_{j\in J}$. The constants $A$ and $B$ are called
$g$-frame bounds. If $A=B$, the $g$-frame is said to be tight
$g$-frame. For more information on frame theory,
basic properties of the $K$-frames and g-frames, we refer to \cite{Christensen.book, Gavruta,Sun}.
Now, we state the following basic definition, the concept of
$K$-$g$-frames, which are more general than ordinary $g$-frames
stated in \cite[Theorem (2.5)]{Asgari.Rahimi}.
\begin{defn}
Let $K \in \mathcal{B}(\mathcal{H})$ and $\Lambda_j \in
\mathcal{B}(\mathcal{H}, \mathcal{H}_j)$ for any $j\in J$. A
sequence $\{\Lambda_j\}_{j\in J}$ is called a K-g-frame for
$\mathcal{H}$ with respect to $\{H_j\}_{j\in J}$, if there exist
constants $0<A \leq B<\infty$ such that
\begin{eqnarray}\label{Def:K-g-frame}
A\|K^{\ast}f\|^2\leq \sum_{j\in J} \|\Lambda_jf\|^2\leq
B\|f\|^2,\,\ \forall f\in \mathcal{H}.
\end{eqnarray}
The constants $A$ and $B$ are called the lower and upper bounds of
K-g-frame, respectively. A $K$-$g$-frame $\{\Lambda_j\}_{j\in J}$
is said to be tight if there exists a constant $A>0$ such that
\begin{eqnarray}
\sum_{j\in J}\|\Lambda_jf\|^2=A\|K^{\ast}f\|^2,\,\ \forall f\in
\mathcal{H}.
\end{eqnarray}
\end{defn}
\noindent {\bf Remark 2:}
If $K=I_{\mathcal{H}}$, then $K$-$g$-frames are just the ordinary $g$-frames.

%==============================================================

\section{\bf New $K$-$g$-frames}

There is an advantage of studying $K$-$g$-frames in comparison with the
$g$-frames. This advantage is, as we will see in the following
examples, for a Bessel sequence in $\mathcal{B}(\mathcal{H},
\mathcal{H}_j)$, which is not an $g$-frame, we can define a
suitable operator $K$ such that this sequence be a $K$-$g$-frame
for $\mathcal{H}$ with respect to $\{\mathcal{H}_j\}_{j\in J}$.
Also, we construct new $K$-$g$-frames by
considering a $g$-frame for $\{\mathcal{H}_j\}_{j\in J}$.\\

\noindent {\bf Example 1:}
\label{Exam.1}
Let $\{e_j\}_{j=1}^{\infty}$ be an orthonormal basis for
$\mathcal{H}$ and $\mathcal{H}_j=\overline{span}\{e_j, e_{j+1}\}$, \,\ $j=1,2,3,\cdots$.
Define the operator
 $\Lambda_j:\mathcal{H}\rightarrow \mathcal{H}_j$ as
follows:
\begin{eqnarray*}
\Lambda_jf=\langle f, e_j+e_{j+1} \rangle(e_j+e_{j+1}),
\,\ \forall f\in \mathcal{H}.
\end{eqnarray*}\
Then,
\begin{eqnarray*}
\sum_{j=1}^{\infty}\|\Lambda_j f\|^2 &=&
2\sum_{j=1}^{\infty}|\langle
f, e_j+e_{j+1} \rangle|^2\\
& \leq & 2\sum_{j=1}^{\infty}(|\langle f, e_j \rangle|+|\langle f, e_{j+1} \rangle|)^2\\
& \leq & 4 \sum_{j=1}^{\infty}|\langle f, e_j \rangle|^2+4
\sum_{j=1}^{\infty}|\langle f, e_{j+1} \rangle|^2\\
& \leq & 8 \|f\|^2.
\end{eqnarray*}
That is, $\{\Lambda_j\}_{j\in J}$ is a $g$-Bessel sequence.
However, $\{\Lambda_j\}_{j\in J}$ does not satisfy the lower
$g$-frame condition, because if we consider the vectors
$g_m:=\sum_{n=1}^m (-1)^{n+1}e_n,\,\ m\in \mathbb{N}$, then
$\|g_m\|^2=m$, for all $m \in \mathbb{N}$. Fix $m\in \mathbb{N}$,
we see that

\begin{eqnarray*}
\langle g_m, e_j+e_{j+1}\rangle=\left\{
\begin{array}{ll}
0, & j>m \\
(-1)^{m+1}, & j=m \\
0, & j<m
\end{array} \right.
\end{eqnarray*}

Therefore,
\begin{eqnarray*}
\sum_{j=1}^{\infty}\|\Lambda_jg_m\|^2=2\sum_{j=1}^{\infty}|\langle
g_m, e_j+e_{j+1}\rangle|^2=2=\frac{2}{m}\|g_m\|^2, \forall m\in
\mathbb{N},
\end{eqnarray*}
that is, $\{\Lambda_j\}_{j\in J}$ does not satisfy the lower
$g$-frame condition. Now define
\begin{eqnarray*}
K:\mathcal{H} \rightarrow \mathcal{H},\,\
Kf=\sum_{j=1}^{\infty}\langle f, e_j\rangle(e_j+e_{j+1}),\,\
\forall f\in \mathcal{H}.
\end{eqnarray*}
Then $\{\Lambda_j\}_{j\in J}$ is a $K$-$g$-frame for $\mathcal{H}$
with respect to $\{\mathcal{H}_j\}_{j\in J}$.\\

\noindent {\bf Example 2:}
\label{Exam.2}
Let $\{e_j\}_{j=1}^{\infty}$ be an
orthonormal basis for $\mathcal{H}$ and
$\mathcal{H}_j=\overline{span}\{e_{3j-2}, e_{3j-1}, e_{3j}\}$, \,\ $j=1,2,3,\cdots$.
Define the operator $\Lambda_j:\mathcal{H}\rightarrow \mathcal{H}_j$
as follows:\\
\begin{eqnarray*}
\Lambda_1f=\langle f, e_1\rangle e_1+\langle f, e_2\rangle e_2+\langle f, e_3\rangle e_3 \,\ \mbox{and}\,\ \Lambda_jf=0, \,\ \mbox{for}\,\
 j\geq 2,
\end{eqnarray*}
With a simple calculate $\{\Lambda_j\}_{j\in J}$ is not a $g$-frame for
$\mathcal{H}$ with respect to $\mathcal{H}_j$,
because, if we take $f=e_4$, then

\begin{eqnarray*}
\|f\|^2=1 \,\ \mbox{and}\,\ \sum_{j=1}^{\infty}\|\Lambda_jf\|^2=\|\Lambda_1e_4\|^2=0.
\end{eqnarray*}
Define now the operator $K:\mathcal{H} \rightarrow \mathcal{H}$ as follows:
\begin{eqnarray*}
Ke_1=e_1,\,\ Ke_2=e_2 \,\ \mbox{and}\,\ Ke_j=0,\,\  \mbox{for}\,\ j\geq 3.
\end{eqnarray*}
It is easy to see that,
$K^*e_1=e_1,\,\ K^*e_2=e_2$ and $K^*e_j=0, \mbox{for}\,\ j\geq 3$.
 We show that $\{\Lambda_j\}_{j\in J}$ is a $K$-$g$-frame for
$\mathcal{H}$ with respect to $\mathcal{H}_j$. In fact, for any
$f\in \mathcal{H}$, we have
\begin{eqnarray*}
\|K^*f\|^2=\|\sum_{j=1}^{\infty}\langle f, e_j\rangle K^*e_j\|^2=
|\langle f, e_1\rangle|^2+|\langle f, e_2\rangle|^2,
\end{eqnarray*}
and
\begin{eqnarray*}
\sum_{j=1}^{\infty}\|\Lambda_jf\|^2&=&\|\Lambda_1f\|^2=
\|\langle f, e_1\rangle e_1+\langle f, e_2\rangle e_2+\langle f, e_3\rangle e_3\|^2\\
&=&|\langle f, e_1\rangle|^2+|\langle f, e_2\rangle|^2+|\langle f, e_3\rangle|^2 \geq \|K^*f\|^2.
\end{eqnarray*}
Therefore, for any $f\in \mathcal{H}$
\begin{eqnarray*}
\|K^*f\|^2\leq \sum_{j=1}^{\infty}\|\Lambda_jf\|^2 \leq \|f\|^2,
\end{eqnarray*}
as desired.
The next Propositions reads as follows:
\begin{prop}
Let $\{\Lambda_j \in \mathcal{B}(\mathcal{H}, \mathcal{H}_j):j\in
J\}$ be a $K$-g-frame for $\mathcal{H}$ with respect to
$\{\mathcal{H}_j\}_{j\in J}$. Then $\{\Lambda_j\}_{j\in J}$ is a
g-frame for $\mathcal{H}$ with respect to $\{\mathcal{H}_j\}_{j\in
J}$ if $K^{\ast}$ is bounded below.
\end{prop}
\begin{proof}
Since $K^{\ast}$ is bounded below, by definition, there exists a
constant $C>0$ such that,
\begin{eqnarray*}
\|K^{\ast}f\|\geq C\|f\|, \forall f\in \mathcal{H}.
\end{eqnarray*}
Therefore, for all $f\in \mathcal{H}$,
\begin{eqnarray*}
AC^2\|f\|^2\leq A\|K^{\ast}f\|^2\leq \sum_{j\in
J}\|\Lambda_{j}f\|^2 \leq B\|f\|^2,
\end{eqnarray*}
that is, $\{\Lambda_j\}_{j\in J}$ is a $g$-frame for $\mathcal{H}$
with respect to $\{\mathcal{H}_j\}_{j\in J}$.
\end{proof}

\begin{prop}
Let $\{\Lambda_j\}_{j\in J}$ be a tight $K$-$g$-frame for
$\mathcal{H}$ with respect to $\{\mathcal{H}_j\}_{j\in J}$ with
$K$-$g$-frame bound $A_1$. Then $\{\Lambda_j\}_{j\in J}$ is a
tight $g$-frame with $g$-frame bound $A_2$ if and only if
$KK^{\ast}=\frac{A_2}{A_1}I_{\mathcal{H}}$.
\end{prop}
\begin{proof}
Let $\{\Lambda_j\}_{j\in J}$ is a tight $g$-frame with bound
$A_2$, then
\begin{eqnarray*}
\sum_{j \in J}\|\Lambda_jf\|^2=A_2\|f\|^2, \,\ \forall f\in
\mathcal{H}.
\end{eqnarray*}
Since $\{\Lambda_j\}_{j\in J}$ is a tight $K$-$g$-frame with bound
$A_1$, we have $A_1\|K^{\ast}f\|^2=A_2\|f\|^2, \,\ \forall f\in
\mathcal{H}$, that is,
\begin{eqnarray*}
\langle KK^{\star}f, f\rangle=\langle \frac{A_2}{A_1}f, f\rangle,
\,\ \forall f\in \mathcal{H}.
\end{eqnarray*}
The converse is straight forward.
\end{proof}

The following Theorem is analog
of ~\cite[Theorem (3.3)]{Casazza.Kutyniok.Li} to obtain new
$K$-$g$-frame for $\mathcal{H}$:
\begin{thm}
Let $K \in \mathcal{B}(\mathcal{H})$ and $\Lambda_j \in
\mathcal{B}(\mathcal{H}, \mathcal{H}_j)$ for any $j\in J$. Let
$\{\Gamma_{ij} \in \mathcal{B}(\mathcal{H}_j,
\mathcal{H}_{ij}):i\in I_j\}$ be a $g$-frame for $\mathcal{H}_j$
 with bounds $C_j$ and $D_j$, such that
$$0<C=\inf_{j\in J} C_j \leq \sup_{j\in J} D_j=D<\infty,$$ where
$\{\mathcal{H}_{ij}\}_{i\in I_j}$ is a sequence of closed
subspaces of $\mathcal{H}_j$, for all $j\in J$. Then the following
statements are equivalent;
\begin{enumerate}
\item[(i)]$\{\Lambda_j \in \mathcal{B}(\mathcal{H}, \mathcal{H}_j):j\in
J\}$ is a $K$-$g$-frame for $\mathcal{H}$
\item[(ii)]$\{\Gamma_{ij}\Lambda_j \in \mathcal{B}(\mathcal{H},
\mathcal{H}_{ij}): i\in I_j, j\in J\}$ is a $K$-$g$-frame
for $\mathcal{H}$.
\end{enumerate}
\end{thm}
\begin{proof}
$(i)\Rightarrow (ii)$ Let $\{\Lambda_j \in
\mathcal{B}(\mathcal{H}, \mathcal{H}_j):j\in J\}$ be a
$K$-$g$-frame for $\mathcal{H}$ with bounds $A_1, B_1$. Then for
all $f \in \mathcal{H}$ we have
\begin{align*}
\sum_{j\in J}\sum_{i\in I_j}\|\Gamma_{ij}\Lambda_jf\|^2 \leq
\sum_{j\in J}D_j\|\Lambda_jf\|^2\leq DB_1\|f\|^2, \\
\sum_{j\in J}\sum_{i\in I_j}\|\Gamma_{ij}\Lambda_jf\|^2 \geq
\sum_{j\in J}C_j\|\Lambda_jf\|^2 \geq CA_1\|K^{\ast}f\|^2.
\end{align*}
$(ii)\Rightarrow (i)$ Let $\{\Gamma_{ij}\Lambda_j \in
\mathcal{B}(\mathcal{H}, \mathcal{H}_{ij}): i\in I_j, j\in J\}$ be
a $K$-$g$-frame for $\mathcal{H}$ with bounds $A_2,B_2$. we have
\begin{align*}
\sum_{j\in J}\|\Lambda_jf\|^2 \leq \sum_{j\in
J}\frac{1}{C_i}\sum_{i\in I_j}\|\Gamma_{ij}\Lambda_jf\|^2 \leq
\frac{B_2}{C} \|f\|^2, \\ \sum_{j\in J}\|\Lambda_jf\|^2 \geq
\sum_{j\in J}\frac{1}{D_j}\sum_{i\in
I_j}\|\Gamma_{ij}\Lambda_jf\|^2 \geq \frac{A_2}{D}\|K^{\ast}f\|^2.
\end{align*}
\end{proof}

Recall that, a sequence $\{w_j\}_j$ is called semi-normalized if
there are bounds $b\geq a> 0$, such that $a \leq |w_j|\leq b$.

\begin{prop}
\label{prop.:semi-normal} Suppose that $\{\Lambda_j\}_{j\in J}$ be
a K-g-frame for $\mathcal{H}$ with respect to
$\{\mathcal{H}_j\}_{j\in J}$ with bounds $A,B$ and $K$-g-dual
$\{\Theta_j\}_{j\in J}$. Let $\{w_j\}_{j\in J}$ be a
semi-normalized sequence with bounds $a, b$. Then
$\{w_j\Lambda_j\}_{j\in J}$ is a K-g-frame with bounds $a^2A$ and
$b^2B$. The sequence $\{w_j^{-1}\Theta_j\}_{j\in J}$ is a K-g-dual
of $\{w_j\Lambda_j\}_{j\in J}$.
\end{prop}

\begin{proof}
Since $\sum_{j\in J}\|w_j\Lambda_jf\|^2 = \sum_{j\in
J}|w_j|^2\|\Lambda_jf\|^2$. Then
\begin{eqnarray*}
a^2\sum_{j\in J}\|\Lambda_jf\|^2 \leq \sum_{j\in J}\|w_j\Lambda_jf\|^2 \leq b^2\sum_{j\in
J}\|\Lambda_jf\|^2,
\end{eqnarray*}
that is, $a^2A\|K^{\ast}f\|^2 \leq \sum_{j\in
J}\|w_j\Lambda_jf\|^2 \leq
b^2B\|f\|^2$.\\
We have $\sum_{j\in J}(w_j\Lambda_j)^{\ast}(w_j^{-1}\Theta_j)f=
\sum_{j\in J}\Lambda_j^{\ast}\Theta_jf=f$. Since $\{w_j^{-1}\}_j$
is bounded, $\{w_j^{-1}\Theta_j\}_j$ is a $g$-Bessel sequence.
Therefore, it is a $K$-$g$-dual of $\{w_j\Lambda_j\}_{j\in J}$.
\end{proof}
\bigskip
Suppose that $\{\Lambda_j\}_{j\in J}$ is a $K$-$g$-frame for
$\mathcal{H}$ with respect to $\{\mathcal{H}_j\}_{j\in J}$.
Obviously, it is a $g$-Bessel sequence, so we can define the
bounded linear operator
$T_{\Lambda}:\ell^2\left(\{\mathcal{H}_j\}_{j\in
J}\right)\rightarrow \mathcal{H}$ as follows:
\begin{eqnarray}
T_{\Lambda}(\{g_j\}_{j\in J})=\sum_{j\in J}\Lambda_j^{\ast}g_j,
\,\ \forall \{g_j\}_{j\in J}\in \ell^2
\left(\{\mathcal{H}_j\}_{j\in J}\right).
\end{eqnarray}
The operator $T_{\Lambda}$ is called the synthesis operator (or
pre-frame operator) for the $K$-$g$-frame $\{\Lambda_j\}_{j\in
J}$. The adjoint operator

\begin{eqnarray}
T_{\Lambda}^{\ast}: \mathcal{H}\rightarrow
l^2(\{\mathcal{H}_j\}_{j\in J}),\quad\
T_{\Lambda}^{\ast}f=\{\Lambda_jf\}_{j\in J}, \,\ \forall f\in
\mathcal{H},
\end{eqnarray}
is called analysis operator for the $K$-$g$-frame
$\{\Lambda_j\}_{j\in J}$. Let
$S_{\Lambda}=T_{\Lambda}T_{\Lambda}^{\ast}$, we obtain the frame
operator for the $K$-$g$-frame $\{\Lambda_j\}_{j\in J}$ as follows
\begin{eqnarray}
S_{\Lambda}:\mathcal{H}\rightarrow \mathcal{H},\quad\
S_{\Lambda}f=\sum_{j\in J}\Lambda_j^{\ast}\Lambda_jf, \,\ \forall
f \in \mathcal{H}.
\end{eqnarray}
%It is well known that, the frame operator $S_{\Lambda}$, need not
%be invertible on $\mathcal{H}$ in general, but it is invertible on
%subspace $\mathcal{R}(K)\subset \mathcal{H}$.
One of the main problem in the $K$-$g$-frame theory is that the
interchangeability of two $g$-Bessel sequence with respect to a
$K$-$g$-frame is different from a $g$-frame. The following
characterization of $K$-$g$-frames is in \cite[Theorem
(2.5)]{Asgari.Rahimi}.
\begin{prop}
Let $K \in \mathcal{B}(\mathcal{H})$. Then the following are
equivalent:
\begin{enumerate}
\item[(i)] $\{\Lambda_j\}_{j\in J}$ is a $K$-$g$-frame for
$\mathcal{H}$ with respect to $\{\mathcal{H}_j\}_{j\in J}$;
\item[(ii)] $\{\Lambda_j\}_{j\in J}$ is a $g$-Bessel sequence for
$\mathcal{H}$ with respect to $\{\mathcal{H}_j\}_{j\in J}$
and there exists a $g$-Bessel sequence
$\{\Gamma_j\}_{j\in J}$ for $\mathcal{H}$ with respect to
$\{\mathcal{H}_j\}_{j\in J}$ such that
\begin{eqnarray}\label{eqn:g-frame.1}
Kf=\sum_{j\in J}\Lambda_j^{\ast}\Gamma_jf,\,\
\forall f\in \mathcal{H}.
\end{eqnarray}
\end{enumerate}
\end{prop}
The position of two $g$-Bessel sequence $\{\Lambda_j\}_{j\in J}$
and $\{\Gamma_j\}_{j\in J}$ in (\ref{eqn:g-frame.1}) are not
interchangeable in general, but there exists another type of dual
such that $\{\Lambda_j\}_{j\in J}$ and a sequence derived by
$\{\Gamma_j\}_{j\in J}$ are interchangeable on the subspace
$\mathcal{R}(K)$ of $\mathcal{H}$.

For $K\in \mathcal{B}(\mathcal{H})$, let $\mathcal{R}(K)$ is
closed, then the pseudo-inverse $K^{\dagger}$ of $K$ exists.
\begin{thm}\cite[Theorem (3.3)]{Hua.Huang}
\label{Thm.rep.1}
Suppose that $\{\Lambda_j\}_{j\in J}$ and
$\{\Gamma_j\}_{j\in J}$ are g-Bessel sequence as in
(\ref{eqn:g-frame.1}). Then there exists a sequence
$\{\Theta_j\}_{j\in J}$ such that
\begin{eqnarray}\label{eqn:g-frame.2}
f=\sum_{j\in J}\Lambda_j^{\ast}\Theta_jf, \,\ \forall f \in
\mathcal{R}(K),
\end{eqnarray}
where $\Theta_j=\Gamma_j(K^{\dagger}\mid_{\mathcal{R}(K)})$.
Moreover, $\{\Lambda_j\}_{j\in J}$ and $\{\Theta_j\}_{j\in J}$ are
interchangeable for any $f \in \mathcal{R}(K)$, that is,
\begin{eqnarray}\label{eqn:g-frame.3}
f=\sum_{j\in J}\Theta_j^{\ast}\Lambda_jf, \,\ \forall f \in
\mathcal{R}(K).
\end{eqnarray}
\end{thm}

%==============================================================

\section{\bf Approximate $K$-$g$-duals}

Motivated by the concept of approximate dual
frames in ~\cite{Christensen.Laugesen}, we will define and focus
on the approximate dual $K$-$g$-frames for $\mathcal{H}$ with
respect to $\{H_j\}_{j\in J}$.

By the Theorem \ref{Thm.rep.1}, since
$K^{\dagger}\mid_{\mathcal{R}(K)}:\mathcal{R}(K)\rightarrow
\mathcal{H}$, we obtain $\Theta_j:\mathcal{R}(K)\rightarrow
\mathcal{H}_j$. For any $f \in \mathcal{R}(K)$, we have
\begin{eqnarray}
\sum_{j\in J}\|\Theta_jf\|^2= \sum_{j\in J}\|\Gamma_j
K^{\dagger}f\|^2\leq B\|K^{\dagger}f\|^2\leq
B\|K^{\dagger}\|^2\|f\|^2.
\end{eqnarray}
That is, $\{\Theta_j\}_{j\in J}$ is a $g$-Bessel sequence for
$\mathcal{R}(K)$ with respect to $\{\mathcal{H}_j\}_{j\in J}$. Let
$T_{\Theta}$ be the synthesis operator of $\{\Theta_j\}_{j\in J}$.
Consider two mixed operators $T_{\Lambda}T_{\Theta}^{\ast}$ and
$T_{\Theta}T_{\Lambda}^{\ast}$ as follows:
\begin{eqnarray*}\label{eqn:Appr.2}
T_{\Lambda}T_{\Theta}^{\ast}:\mathcal{R}(K)\rightarrow
\mathcal{H}, \,\
T_{\Lambda}T_{\Theta}^{\ast}f=\sum_{}\Lambda_j^{\ast}\Theta_jf,
\,\ \forall f \in \mathcal{R}(K),
\end{eqnarray*}
\begin{eqnarray*}
T_{\Theta}T_{\Lambda}^{\ast}:\mathcal{H}\rightarrow
\mathcal{R}(K), \,\
T_{\Theta}T_{\Lambda}^{\ast}f=\sum_{j}\Theta_j^{\ast}\Lambda_jf,
\,\ \forall f \in \mathcal{H}.
\end{eqnarray*}
By the Theorem \ref{Thm.rep.1} and similar to the definition of
duality and approximate duality stated in
~\cite{Christensen.Laugesen}, we have the following definition:

\begin{defn}
Consider two g-Bessel sequences $\{\Lambda_j \in
\mathcal{B}(\mathcal{H},\mathcal{H}_j): j\in J\}$ and $\{\Theta_j
\in \mathcal{B}(\mathcal{H},\mathcal{H}_j): j\in J\}$.
\begin{enumerate}
\item[(i)] The sequences $\{\Lambda_j\}_{j\in J}$ and
$\{\Theta_j\}_{j\in J}$ are dual K-g-frames, when
$T_{\Lambda}T_{\Theta}^{\ast}=I_{\mathcal{R}(K)}$ or
$T_{\Theta}T_{\Lambda}^{\ast}|_{\mathcal{R}(K)}=I_{\mathcal{R}(K)}$.
In this case, we say that $\{\Theta_j\}_{j\in J}$ is a
K-g-dual of $\{\Lambda_j\}_{j\in J}$,
\item[(ii)] The sequences $\{\Lambda_j\}_{j\in J}$ and
$\{\Theta_j\}_{j\in J}$ are approximately dual K-g-frames,
whenever $\|I_{\mathcal{R}(K)}-T_{\Lambda}T_{\Theta}^{\ast}\|<1$ or
$\|I_{\mathcal{R}(K)}-T_{\Theta}T_{\Lambda}^{\ast}|_{\mathcal{R}(K)}\|<1$.
In this case, we say that $\{\Theta_j\}_{j\in J}$ is an
approximate dual K-g-frame or approximate K-g-dual of
$\{\Lambda_j\}_{j\in J}$.
\end{enumerate}

\end{defn}
Note that the given conditions are equivalent, by taking adjoint.
A well-known algorithm to find the inverse of an operator is the
Neumann series algorithm.\\ Since
$\|I_{\mathcal{R}(K)}-T_{\Lambda}T_{\Theta}^{\ast}\|<1$,
$T_{\Lambda}T_{\Theta}^{\ast}$ is invertible with
$$(T_{\Lambda}T_{\Theta}^{\ast})^{-1}=
(I_{\mathcal{R}(K)}-(I_{\mathcal{R}(K)}-T_{\Lambda}T_{\Theta}^{\ast}))^{-1}
=\sum_{n=0}^{\infty}(I_{\mathcal{R}(K)}-T_{\Lambda}T_{\Theta}^{\ast})^n.$$
Therefore, every $f \in \mathcal{R}(K)$ can be reconstruct as

\begin{eqnarray}\label{eqn:Appr.3}
f=\sum_{n=0}^{\infty}(I_{\mathcal{R}(K)}-T_{\Lambda}T_{\Theta}^{\ast})^n
T_{\Lambda}T_{\Theta}^{\ast}f.
\end{eqnarray}
The following Propositions is the analog of Prop. 3.4 and Prop.
4.1 in ~\cite{Christensen.Laugesen} to obtain new approximate
$K$-$g$-duals.
\begin{prop}
Let $\{\Theta_j\}_{j\in J}$ be an approximate $K$-$g$-dual of
$\{\Lambda_j\}_{j\in J}$, then
$\{\Theta_j(T_{\Lambda}T_{\Theta}^{\ast})^{-1}\}$ is a K-g-dual of
$\{\Lambda_j\}_{j\in J}$.
\end{prop}
\begin{proof}
It is easy to see that
$\{\Theta_j(T_{\Lambda}T_{\Theta}^{\ast})^{-1}\}_{j\in J}$ is a
$g$-Bessel sequence and
\begin{eqnarray*}
f=(T_{\Lambda}T_{\Theta}^{\ast})(T_{\Lambda}T_{\Theta}^{\ast})^{-1}f
& = &
\sum_{j=0}^{\infty}\Lambda_j^{\ast}\Theta_j(T_{\Lambda}T_{\Theta}^{\ast})^{-1}f\\
& = &
\sum_{j=0}^{\infty}\Lambda_j^{\ast}(\Theta_j\sum_{n=0}^{\infty}
(I_{\mathcal{R}(K)}-T_{\Lambda}T_{\Theta}^{\ast})^nf).
\end{eqnarray*}
Therefore,
$\{\Theta_j(T_{\Lambda}T_{\Theta}^{\ast})^{-1}\}=\{\Theta_j\sum_{n=0}^{\infty}
(I_{\mathcal{R}(K)}-T_{\Lambda}T_{\Theta}^{\ast})^n\}_{j\in J}$ is
a $K$-$g$-dual of $\{\Lambda_j\}_{j\in J}$.
\end{proof}
For each $N\in \mathbb{N}$, define
$\gamma_j^{(N)}=\sum_{n=0}^N\Theta_j(I_{\mathcal{R}(K)}-T_{\Lambda}
T_{\Theta}^{\ast})^n$. Define $T_N:\mathcal{H}\rightarrow
\mathcal{H}$ by
$T_N=\sum_{n=0}^N(I_{\mathcal{R}(K)}-T_{\Lambda}T_{\Theta}^{\ast})^n$,
Then $\gamma_j^{(N)}=\Theta_jT_N, \,\ \forall j\in J$. The
sequence $\{\gamma_j^{(N)}\}_{j \in J}$ is obtained from the
$g$-Bessel sequence $\{\Theta_j\}_{j\in j}$ by a bounded operator,
therefore, it is a $g$-Bessel sequence. For each $f \in
\mathcal{R}(K)$,
\begin{eqnarray*}
T_{\Lambda}T_{\Theta}^{\ast}T_Nf=
\sum_{j=0}^{\infty}\Lambda_j^{\ast}\Theta_jT_Nf=
\sum_{j=0}^{\infty}\Lambda_j^{\ast}\gamma_j^{(N)}f=
T_{\Lambda}T_{\Gamma}^{\ast}f,
\end{eqnarray*}
where $T_{\Gamma}$ is the synthesis operator of
$\{\gamma_j^{(N)}\}_{j \in J}$. Thus,
\begin{align*}
T_{\Lambda}T_{\Gamma}^{\ast}f &=T_{\Lambda}T_{\Theta}^{\ast}T_Nf\\
&=[I_{\mathcal{R}(K)}-(I_{\mathcal{R}(K)}-T_{\Lambda}T_{\Theta}^{\ast})]\sum_{n=0}^N
(I_{\mathcal{R}(K)}-T_{\Lambda}T_{\Theta}^{\ast})^nf\\
&=[I_{\mathcal{R}(K)}-(I_{\mathcal{R}(K)}-T_{\Lambda}T_{\Theta}^{\ast})^{N+1}]f,
\end{align*}
by telescoping. Thus
\begin{align}\label{eqn:Appr.7}
\|I_{\mathcal{R}(K)}-T_{\Lambda}T_{\Gamma}^{\ast}\|&=\|(I_{\mathcal{R}(K)}-T_{\Lambda}
T_{\Theta}^{\ast})^{N+1}\|\\
& \leq\|I_{\mathcal{R}(K)}-T_{\Lambda}T_{\Theta}^{\ast})\|^{N+1}.
\end{align}
If $\{\Theta_j\}_{j\in J}$ be an approximate K-g-dual of
$\{\Lambda_j\}_{j\in J}$, then
\begin{align}\label{eqn:Appr.8}
\|I_{\mathcal{R}(K)}-T_{\Lambda}T_{\Theta}^{\ast})\|<1.
\end{align}
By (\ref{eqn:Appr.7}) and (\ref{eqn:Appr.8}), we obtain
\begin{align*}
\|I_{\mathcal{R}(K)}-T_{\Lambda}T_{\Gamma}^{\ast})\|<1,
\end{align*}
That is, $\{\gamma_j^{(N)}\}_{j\in J}$ is an approximate
$K$-$g$-dual of $\{\Lambda_j\}_{j\in J}$.\\

We now summarize what we have proved:

\begin{prop}
Let $\{\Theta_j\}_{j\in J}$ be an approximate K-g-dual of
$\{\Lambda_j\}_{j\in J}$. Then
\begin{align*}
\{\sum_{n=0}^{N}\Theta_j(I_{\mathcal{R}(K)}-T_{\Lambda}T_{\Theta}^{\ast})^n\}_{j\in
J}
\end{align*}
is an approximate K-g-dual of $\{\Lambda_j\}_{j\in J}$.
\end{prop}

\begin{thm}
Let $\{\Lambda_j \in \mathcal{B}(\mathcal{H},\mathcal{H}_j): j\in
J\}$ be a $K$-$g$-frame and $\{\Theta_j \in
\mathcal{B}(\mathcal{H},\mathcal{H}_j): j\in J\}$ be a $g$-Bessel
sequence. Let also $\{f_{i,j}\}_{i\in I_j}$ be a frame for
$\mathcal{H}_j$ with bounds $A_j$ and $B_j$ for every $j\in J$
such that, $0<A=\inf_{j\in J}A_j\leq \sup_{j\in J}B_j=B<\infty$.
Then $\{\Theta_j\}_{j\in J}$ is an approximate $K$-$g$-dual of
$\{\Lambda_j\}_{j\in J}$ if and only if
$E=\{\Theta_j^{\ast}f_{i,j}\}_{i\in I_j, j\in J}$ is an
approximate dual of
$F=\{\Lambda_j^{\ast}\widetilde{f}_{i,j}\}_{i\in I_j, j\in J}$,
where $\{\widetilde{f}_{i,j}\}_{i \in I_j}$ is the canonical dual
of $\{f_{i,j}\}_{i \in I_j}$.
\end{thm}
\begin{proof}
For each $f\in \mathcal{H}$ we have
\begin{align*}
\sum_{j\in J}\sum_{i\in I_j}|\langle f, \Theta_j^{\ast}f_{i,j}
\rangle|^2 & =\sum_{j\in J}\sum_{i\in I_j}|\langle \Theta_jf,
f_{i,j}
\rangle|^2\\
& \leq \sum_{j\in J}B_j\|\Theta_jf\|^2 \leq B\sum_{j\in
J}\|\Theta_jf\|^2.
\end{align*}
This implies that $E$ is a Bessel sequence for $\mathcal{H}$.
Similarly, $F$ is also Bessel sequence for $\mathcal{H}$.
Moreover, for each $f\in \mathcal{H}$ we have
\begin{align*}
T_{\Theta}T_{\Lambda}^{\ast}f & =
\sum_{j\in J}\Theta_j^{\ast}\Lambda_jf\\
& =\sum_{j\in J}\Theta_j^{\ast}\sum_{i\in I_j}\langle \Lambda_jf,
\widetilde{f}_{i,j} \rangle f_{i,j}\\
& =\sum_{j\in J,i\in I_j}\langle f,
\Lambda_J^{\ast}\widetilde{f}_{i,j} \rangle
\Theta_j^{\ast}f_{i,j}\\
& =T_ET_F^{\ast}f.
\end{align*}
So $\|I-T_{\Theta}T_{\Lambda}^{\ast}\|<1$ if and only if
$\|I-T_ET_F^{\ast}\|<1$, this follows the result.
\end{proof}

%============================================================

\section{\bf Redundancy}

One of the important property of frame theory is the possibility
of redundancy. For example, in ~\cite[Theorem
(3.2)]{Casazza.Kutyniok} the authors have provided sufficient
conditions on the weights in the fusion frames to be a fusion
frames, when some elements erasure, that is, some arbitrary
elements can be removed without destroying the
fusion frame property of the remaining set.\\
Our main result in this section provides an equivalent condition
under which the subsequence of a $K$-$g$-frame still to be a
$K$-$g$-frame. In the theory of $K$-$g$-frames, if we have
information on the lower $K$-$g$-frame bound and the norm of the
$K$-$g$-frame elements, we can provide a criterion for how many
elements we can remove:

\begin{prop}
Let $K\in \mathcal{B}(\mathcal{H})$, such that $K^{\ast}$ is
bounded below with constant $C>0$ and also let $\{\Lambda_j \in
\mathcal{B}(\mathcal{H}, \mathcal{H}_j):j\in J\}$ is a
$K$-$g$-frame with lower bound $A>\frac{1}{C}$. Then for each
subset $I\subset J$ with $|I|<AC$ such that, $\|\Lambda_j\|=1,
\forall j \in I$, the family $\{\Lambda_j\}_{j\in J\backslash I}$
is a $K$-$g$-frame for $\mathcal{H}$ with respect to
$\{\mathcal{H}_j\}_{j\in J}$ with lower $K$-$g$-frame bound
$AC^2-|I|$.
\end{prop}
\begin{proof}
Given $f \in \mathcal{H}$,
\begin{align*}
\sum_{j\in I}\|\Lambda_j f\|^2 \leq \sum_{j\in I}\|\Lambda_j\|^2
\|f\|^2=|I|\|f\|^2.
\end{align*}
Thus
\begin{align*}
\sum_{j\in J\backslash I}\|\Lambda_j f\|^2 & \geq
A\|K^{\ast}f\|^2-|I|\|f\|^2\\
& \geq AC^2\|f\|^2-|I|\|f\|^2\\
& =(AC^2-|I|)\|f\|^2.
\end{align*}
\end{proof}

\begin{thm}
Let $I\subset J$. Suppose that $\{\Lambda_j \in
\mathcal{B}(\mathcal{H}, \mathcal{H}_j):j\in J\}$ is a
$K$-$g$-frame with bounds $A, B$ and $K$-$g$-frame operator
$S_{\Lambda,J}$. Then the following statements are equivalent:
\begin{enumerate}
\item[(i)]
$I-S^{-1}_{\Lambda,J}S_{\Lambda,I}$ is boundedly invertible,
\item[(ii)] The sequence $\{\Lambda_j\}_{j\in J\backslash I}$ is a $K$-$g$-frame
for $\mathcal{H}$ with respect to $\{\mathcal{H}_j\}_{j\in J}$
with lower $K$-$g$-frame bound
$\frac{A}{\|(I-S^{-1}_{\Lambda,J}S_{\Lambda,I})^{-1}\|^2}$.
\end{enumerate}
\end{thm}
\begin{proof}
Denote the frame operator of $K$-$g$-frame $\{\Lambda_j\}_{j\in
J\backslash I}$ by $S_{\Lambda,J\backslash I}$. Since
$$S_{\Lambda,J\backslash I}=S_{\Lambda,J}-S_{\Lambda,I}=
S_{\Lambda,J}(I-S^{-1}_{\Lambda,J}S_{\Lambda,I}),$$ we have,
$\{\Lambda_j\}_{j\in J \backslash I}$ is $K$-$g$-frame if and only
if $S_{\Lambda,J\backslash I}$ is boundedly invertible and then if
and only if $I-S^{-1}_{\Lambda,J}S_{\Lambda,I}$ is boundedly invertible.\\
Now, for the lower $K$-$g$-frame bound, assume that
$I-S^{-1}_{\Lambda,J}S_{\Lambda,I}$ is invertible. Since
$\{\Lambda_j\}_{j\in J}$ is a $K$-$g$-frame for $\mathcal{H}$ with
respect to $\{\mathcal{H}_j\}_{j\in J}$ with bounds $A$ and $B$,
for any $f \in \mathcal{H}$,
\begin{align*}
f & =S^{-1}_{\Lambda,J}S_{\Lambda,J}f\\
& =S^{-1}_{\Lambda,J}(\sum_{j\in I}\Lambda^{\ast}_j\Lambda_jf +
\sum_{j\in J\backslash
I}\Lambda^{\ast}_j\Lambda_jf)\\
& = S^{-1}_{\Lambda,J}S_{\Lambda,I}f+ \sum_{j\in J\backslash I}
S^{-1}_{\Lambda,J} \Lambda^{\ast}_j\Lambda_jf.
\end{align*}
Hence we have, $(I-S^{-1}_{\Lambda,J}S_{\Lambda,I})f= \sum_{j\in
J\backslash I} S^{-1}_{\Lambda,J} \Lambda^{\ast}_j\Lambda_jf$.

Therefore we obtain
\begin{align*}
\|(I-S^{-1}_{\Lambda,J}S_{\Lambda,I})f\| & =\|\sum_{j\in
J\backslash I} S^{-1}_{\Lambda,J}
\Lambda^{\ast}_j\Lambda_jf\|\\
& = \sup_{g\in \mathcal{H}, \|g\|=1}|\langle \sum_{j\in
J\backslash I} S^{-1}_{\Lambda,J}
\Lambda^{\ast}_j\Lambda_jf, g \rangle|\\
& = \sup_{g\in \mathcal{H}, \|g\|=1}|\sum_{j\in J\backslash I}
\langle\Lambda_jf, \Lambda_jS^{-1}_{\Lambda,J}g \rangle|\\
& \leq \sup_{g\in \mathcal{H}, \|g\|=1}\sum_{j\in J\backslash
I} \|\Lambda_jf\| \|\Lambda_jS^{-1}_{\Lambda,J}g\|\\
& \leq \sup_{g\in \mathcal{H}, \|g\|=1}(\sum_{j\in J\backslash I}
\|\Lambda_jf\|^2)^{\frac{1}{2}}(\sum_{j\in J\backslash I}
\|\Lambda_jS^{-1}_{\Lambda,J}g\|^2)^{\frac{1}{2}}\\
& \leq \sup_{g\in \mathcal{H}, \|g\|=1}(\sum_{j\in J\backslash I}
\|\Lambda_jf\|^2)^{\frac{1}{2}}(\langle
S_{\Lambda,J}(S^{-1}_{\Lambda,J}g),
S^{-1}_{\Lambda,J}g\rangle)^{\frac{1}{2}}.
\end{align*}
That is
\begin{align}\label{Noneqn:1}
\|(I-S^{-1}_{\Lambda,J}S_{\Lambda,I})f\| \leq
\sqrt{A^{-1}}(\sum_{j\in J\backslash I}
\|\Lambda_jf\|^2)^{\frac{1}{2}}.
\end{align}
It follows that $I-S^{-1}_{\Lambda,J}S_{\Lambda,I}$ is well
defined in $\mathcal{H}$. If $I-S^{-1}_{\Lambda,J}S_{\Lambda,I}$
is invertible on $\mathcal{H}$, then for any $f \in \mathcal{H}$
we have
\begin{align}\label{Noneqn:2}
\|K^{\ast}f\|\leq
\|K^{\ast}(I-S^{-1}_{\Lambda,J}S_{\Lambda,I})^{-1}\|\cdot\|(I-S^{-1}_{\Lambda,J}S_{\Lambda,I}f)\|.
\end{align}
From (\ref{Noneqn:1}) and (\ref{Noneqn:2}) we have
\begin{align*}
\frac{A}{\|K^{\ast}(I-S^{-1}_{\Lambda,J}S_{\Lambda,I})^{-1}\|^2}\|K^{\ast}f\|^2
\leq \sum_{j\in J\backslash I} \|\Lambda_jf\|^2,\,\ \forall f \in
\mathcal{H}.
\end{align*}
This complete the proof.
\end{proof}

%==================================================================

\end{document}